\documentclass[<preprint>]{elsarticle}
\usepackage[english]{babel}
\usepackage{amsmath,amsthm}

\usepackage{amsfonts}

\newtheorem{thm}{Theorem}[section]
\newtheorem{cor}[thm]{Corollary}
\newtheorem{lem}[thm]{Lemma}

\theoremstyle{definition}

\theoremstyle{remark}

\numberwithin{equation}{section}
\begin{document}

\begin{frontmatter}

\title{Determinants of Circulant Matrices with Some Certain Sequences}

\author{Ercan Alt\i n\i \c{s}\i k}
\ead{ercanaa@gmail.com, ealtinisik@gazi.edu.tr}
\address{Department of Mathematics, Faculty of Sciences, Gazi University \\ 06500 Teknikokullar - Ankara, Turkey}

\begin{abstract}
Let $\{a_k\}$ be a sequence of real numbers defined by an $m$th order linear homogenous recurrence relation. In this paper we obtain a determinant formula for the circulant matrix $A=circ(a_1, a_2, \cdots, a_n)$, providing a generalization of determinantal results in papers of Bozkurt \cite{Bozkurt}, Bozkurt and Tam \cite{BozkurtTam}, and Shen, et al. \cite{ShenCenHao}.
\end{abstract}

\begin{keyword}
\texttt{circulant matrix, determinant, Fibonacci sequence, Lucas sequence, tribonacci sequence}
\MSC[2010] 15B05, 15A15, 11B39
\end{keyword}

\end{frontmatter}

\section{Introduction}

The circulant matrix $V=circ(v_1, v_2, \ldots ,v_n)$ associated to real numbers $v_1, v_2, \ldots ,v_n$ is the $n \times n$ matrix
\[
        V= \left(
              \begin{array}{cccc}
                v_1 & v_2 & \cdots & v_n \\
                v_n & v_1 & \cdots & v_{n-1} \\
                \vdots & \vdots & \ddots & \vdots \\
                v_2 & v_3 & \ldots & v_1 \\
              \end{array}
        \right).
\]
Circulant matrices are one of the most interesting members of matrices. They have elegant algebraic properties. For example, $Circ(n)$ is an algebra on $\mathbb{C}$. Let $\epsilon$ be a primitive $n^{th}$ root of unity. For each $0 \leq k \leq n-1 $, $\lambda_k = \sum_{j=0}^{n-1} v_j \epsilon^{kj}$ is an eigenvalue of $V=circ(v_1, v_2, \ldots , v_n)$ and the corresponding eigenvector is $ x_k = \frac{1}{\sqrt{n}} (1, \epsilon^k, \epsilon^{2k}, \ldots , \epsilon^{(n-1)k} ) \in \mathbb{C}^n$. Indeed, all circulant matrices have the same ordered set of orthonormal eigenvectors $\{x_k\}.$ Besides, $\det V = \prod_{k=0} ^{n-1} \left( \sum_{j=0}^{n-1} v_j \epsilon^{kj} \right)$. The reader can consult the text of Davis \cite{Davis} for further properties of circulant matrices. On the other hand, circulant matrices have a widespread applications in many parts of mathematics. The excellent survey paper \cite{KraSimanca} includes many applications of circulant matrices in various areas of mathematics. Also, they have applications in signal processing, the study of cyclic codes for error corrections \cite{Gray} and in quantum mechanics \cite{Aldrovandi}.

Recently, many authors have investigated some properties of circulant matrices associated to so famous integer sequences, for example, the Fibonacci sequence and the Lucas sequence. Let $a,b,p,q \in \mathbb{Z}$. Define a sequence $(U_n)$ by the second order recurrence relation
\begin{equation}\label{secondorderrecurrence}
U_n = p U_{n-1} + q U_{n-2}
\end{equation}
$(n \geq 3)$ with initial conditions $U_1=a$ and $ U_2=b$. Taking $(p,q,a,b)=(1,1,1,1)$, $(1,1,1,3)$, $(1,2,1,1)$ and $(1,2,1,3)$, $(U_n)$ becomes the Fibonacci sequence $(F_n)$, the Lucas sequence $(L_n)$, the Jacobsthal sequence $(J_n)$ and the Jacobsthal-Lucas sequence $(j_n)$, respectively. In 1970 Lind \cite{Lind} obtained a formula for the determinant of $F=circ(F_{r}, F_{r+1}, \ldots, F_{r+n-1})$ $(r \geq 1)$. In 2005 Solak \cite{Solak} investigated matrix norms of $F=circ(F_1, F_2, \ldots, F_n)$ and $L=circ(L_1, L_2, \ldots, L_n)$. In 2011 Shen, Cen and Hao \cite{ShenCenHao} showed that
\[
\det (F) = (1-F_{n+1})^{n-1} + F_{n}^{n-2} \sum_{k=1}^{n-1} F_{k} \left( \frac{1-F_{n+1}}{F_{n}}\right)^{k-1}
\]
and
\[
\det (L) = (1-L_{n+1})^{n-1} + (L_{n}-2)^{n-2} \sum_{k=1}^{n-1} (L_{k+2}-3L_{k+1}) \left( \frac{1-L_{n+1}}{L_{n}-2}\right)^{k-1}
\]
Recently, Bozkurt and Tam \cite{BozkurtTam} have obtained determinant formulae for $J=circ(J_1, J_2, \ldots, J_n)$ and $\mathbb{J}=circ(j_1, j_2, \ldots, j_n)$ using the same method. Then Bozkurt \cite{Bozkurt} has given a generalization of these determinant formulae as
\begin{eqnarray}\label{bozkurtequation}
\det (U) &=& (a^2-b U_{n})(a- U_{n+1})^{n-2} \nonumber \\
& & + \sum_{k=2}^{n-1}(a U_{k+1} - b U_k) (a-U_{n+1})^{k-2} (q U_n-b+qa)^{n-k},
\end{eqnarray}
where $\{ U_k \}$ is the sequence in (\ref{secondorderrecurrence}).

In all of the above-mentioned papers authors calculated determinants of circulant matrices associated to a sequence defined by a second order recurrence relation by using the same method. In this paper we generalize determinantal results of these papers for certain sequences defined by a recurrence relation of order $m \geq 1$.

\section{The Main Result}
Let $c_1,c_2, \ldots ,c_m$ be real numbers and $c_m \neq 0$. Consider the sequence $\{a_k\}$ defined by the $m$th order linear homogenous recurrence relation

\begin{equation} \label{recurrencerelation}
a_k = c_1 a_{k-1} + c_2 a_{k-1} +\cdots + c_m a_{k-m} \qquad  (k\geq m+1)
\end{equation}
with initial conditions
\begin{equation} \label{initialconditions}
a_1, a_2 ,\ldots, a_m,
\end{equation}
which are given real numbers. Let $n>m$ and $A=circ(a_1,a_2,\ldots, a_n)$. Let $A_{ij}$ be the $ij-$entry of $A$. It is clear that $A_{ij}=a_{j-i+1}$ if $j\geq i$ and $a_{n+j-i+1}$ otherwise. On the other hand, for simplicity, we write $ A_{ij} = a_{(j-i+1)}$ in both case. Our main goal is to reduce the order $n$ of the determinant of $A$ and to calculate it in a simpler way. In order to perform this, first we define an $n \times n$ matrix $P=(P_{ij})$, where
\[
P_{ij}=\left\{ \begin{array}{cl}
           1 & \textrm{if  } i=j=1 \textrm{ or } i+j=n+2, \\
           -c_m & \textrm{if  } i=m+1 \textrm{ and } j=1, \\
           -c_t & \textrm{if  } i+j-t=n+2 \textrm{ and } i \geq m+1 \textrm{ and } 1 \leq t \leq m , \\
           0 & \textrm{otherwise,}
         \end{array}\right.
\]
Then the $ij-$entry of the product of $P$ and $A$ is
\[
(PA)_{ij}= \left\{ \begin{array}{cl}
           A_{1j} & \textrm{if  } i=1, \\
           A_{n-i+2,j} & \textrm{if  } 2 \leq i \leq m , \\
           \alpha_t & \textrm{if  } i+j=n+t+1 \textrm{ and } 1 \leq t \leq m, \\
           0 & \textrm{otherwise,}
         \end{array} \right.
\]
where
\begin{eqnarray} \label{alphats}
\alpha_t &=& A_{n-m+1,n-m+t} - c_1 A_{n-m+2,n-m+t} \nonumber \\
& & - \cdots - c_{m-1} A_{n,n-m+t} - c_m A_{1,n-m+t}.
\end{eqnarray}
Now, we define a sequence $\{b_{s}^{(r)}\}$ for every $r=1,2,\ldots, m-1$ by the recurrence relation
\begin{equation} \label{bsrecurence}
b_{s}^{(r)} = - \frac{\alpha_2}{\alpha_1} b_{s-1}^{(r)}  - \frac{\alpha_3}{\alpha_1} b_{s-2}^{(r)} - \cdots  - \frac{\alpha_{m}}{\alpha_1} b_{s-m+1}^{(r)} \qquad  (s \geq m)
\end{equation}
with initial conditions
\begin{equation} \label{bsinitialconditions}
b_{i}^{(r)}=\delta_{i,r},
\end{equation}
the Kronecker delta, for $i=1,2,\ldots, m-1$. We form another $n \times n$ matrix $Q=(Q_{ij})$ such that
\[
Q_{ij} =\left\{ \begin{array}{cl}
           1 & \textrm{if  } i=j=1 \textrm{ or } i+j=n+2, \\
           b_{n-i+1}^{(j-1)} & \textrm{if  } 2 \leq i \leq n-m+1 \textrm{ and } 2 \leq j \leq m, \\
           0 & \textrm{otherwise.}
         \end{array}\right.
\]
Then, we have
\[
(PAQ)_{ij}= \left\{ \begin{array}{cl}
           A_{1,1} & \textrm{if  } i=j=1, \\
           A_{n-i+2,1} & \textrm{if  } 2 \leq i \leq m \textrm{ and } j=1, \\
           \sum_{k=2}^n A_{1k}b_{n-k+1}^{(j-1)} & \textrm{if  } i=1 \textrm{ and } 2 \leq j \leq m, \\
           \sum_{k=2}^n A_{n-i+2,k}b_{n-k+1}^{(j-1)} & 2 \leq i, j \leq m, \\
           \alpha_k & \textrm{if  } i, j > m \textrm{ and } 1 \leq k \leq m \textrm{ and } i-j=k-1, \\
           0 & \textrm{otherwise.} \\
         \end{array} \right.
\]
Recall that $A_{ij}=a_{j-i+1}$ if $j \geq i$ and $a_{n+j-i+1}$ otherwise and that we write $A_{i,j}=a_{(j-i+1)}$ for simplicity. Also, it is clear that $\det P = \det Q = (-1)^{\frac{n(n+1)}{2}-1}$ and $\alpha_1=a_1-a_{n+1}$. Finally, we get the following lemma.
\begin{lem}\label{main}
Let $\{a_k\}$ be the sequence defined by the recurrence relation in (\ref{recurrencerelation}) with initial conditions in (\ref{initialconditions}), $n>m$ and $A=circ(a_1,a_2,\ldots, a_n)$. Then
\begin{equation} \label{detgeneral}
\scriptsize
\setlength{\arraycolsep}{.25\arraycolsep}
\det(A)= (a_1-a_{n+1})^{n-m} \sum_{k_1=2}^n \cdots \sum_{k_{m-1}=2}^n \left|
                                                                   \begin{array}{cccc}
                                                                     a_1    & a_{(k_1)}     & \cdots & a_{(k_{m-1})} \\
                                                                     a_2    & a_{(k_1+1)}   & \cdots & a_{(k_{m-1}+1)} \\
                                                                     \vdots & \vdots            &        & \vdots \\
                                                                     a_m    & a_{(k_1+m-1)} & \cdots &  a_{(k_{m-1}+m-1)} \\
                                                                   \end{array}
                                                                 \right| \prod_{i=1}^{m-1} b_{n-k_i+1}^{(i)},
\normalsize
\end{equation}
where sequences $\{b_{s}^{(r)}\}$ are defined by the recurrence relation in (\ref{bsrecurence}) with initial conditions in (\ref{bsinitialconditions}).
\end{lem}

Indeed, the determinant formula for $A=circ(a_1, a_2, \ldots, a_n)$ in Lemma~\ref{main} is not effective but we obtain it by generalizing the common method of papers \cite{ShenCenHao,BozkurtTam,Bozkurt} for the sequence $ \{a_k \}$ defined by a recurrence relation of order $m \geq 1$. To illustrate our goal we consider the well-known tribonacci sequence. The tribonacci sequence $\{ a_k \}$ is defined by the recurrence relation
\[
a_k=a_{k-1}+a_{k-2}+a_{k-3} \qquad (k \geq 4)
\]
with initial conditions $a_1=1$, $a_2=1$, $a_3=2$. For convenience, we take $a_0=0$.
\begin{cor} \label{tribonaccidet}
Let $\{ a_k \}$ be the tribonacci sequence, $n>3$ and $A=circ(a_1, a_2, \ldots, a_n)$. Then
\begin{eqnarray*}
\det(A)&=&
(1-a_{n+1})^{n-3}\bigg(\sum_{i=2}^{n-3}\sum_{j=i+1}^{n-2}\big(a_{i-2}a_{j-1}-a_{i-1}a_{j-2}\big)\big(\frac{\alpha_3}{\alpha_1}\big)^{n-j-1}b_{j-i+2}^{(1)} \\ & & + \sum_{i=2}^{n-2}\big((a_{i-2}+a_{i-1})+a_{n-1}(a_{i+2}-2a_{i+1})+a_{n}(2a_{i}-a_{i+2})\big)b_{n-i+1}^{(1)} \\
& & + \sum_{i=2}^{n-2}\big(-a_{i-1}+a_{n}(a_{i+2}-2a_{i+1})\big) \frac{\alpha_1}{\alpha_3} b_{n-i+2}^{(1)} + ( 2a_n^2 - 2a_{n} - a_{n-1} + 1 )\bigg).
\end{eqnarray*}

\end{cor}
\begin{proof}
Let $\{ a_k \}$ in Lemma~\ref{main} be the tribonacci sequence. Then clearly $m=3$, $a_1=a_2=1$, $a_3=2$, $\alpha_1=1-a_{n+1}$ and by Lemma~\ref{main}, we have
$$
\det(A)= (1-a_{n+1})^{n-3} \sum_{i=2}^{n} \sum_{j=2}^{n} \left|
                                                                   \begin{array}{ccc}
                                                                     1   & a_{(i)}    & a_{(j)} \\
                                                                     1   & a_{(i+1)}  & a_{(j+1)} \\
                                                                     2   & a_{(i+2)}  & a_{(j+2)} \\
                                                                   \end{array}
                                                                 \right|
                                                                 b_{n-i+1}^{(1)} b_{n-j+1}^{(2)}.
$$
We denote the $3 \times 3$ determinant in the summation by $ \Delta ((i),(j))$. It is clear that $ \Delta ((i),(i))=0$ and $ \Delta ((j),(i))= - \Delta ((i),(j))$.  Also, we have $\Delta ((i),(j)) = \Delta (i,j)$ if $1 \leq i,j \leq n-3$. Thus
$$
\det(A)= (1-a_{n+1})^{n-3} \sum_{i=2}^{n-1} \sum_{j=i+1}^{n} \Delta ((i),(j))  \big( b_{n-i+1}^{(1)} b_{n-j+1}^{(2)} - b_{n-j+1}^{(1)} b_{n-i+1}^{(2)} \big).
$$
Now, sequences $\{b_{k}^{(1)}\}$ and $\{b_{k}^{(2)}\}$ are generated by the recurrence relation in (\ref{bsrecurence}) with different initial conditions, all of which are given in (\ref{bsinitialconditions}). The characteristic equation of the recurrence relation in (\ref{bsrecurence}) is $ \alpha_1 r^2 + \alpha_2 r + \alpha_3 =0 $, where $\alpha_1 = 1- a_{n+1}$, $\alpha_2 = -a_{n}- a_{n-1}$ and $\alpha_3 = -a_{n}$. Since $ \alpha_2^2-4\alpha_1 \alpha_3 < (-a_n + a_{n-1})(3a_n + a_{n-1}) < 0 $ for all $n \geq 1$, the characteristic equation has two distinct complex roots, say $\lambda$ and $\mu$. Finally, Binet's formulae for sequences $b_{k}^{(1)}$ and $b_{k}^{(2)}$ are $ b_{k}^{(1)} = \frac{\lambda \mu}{\mu - \lambda}\big( \lambda^{k-2} -\mu^{k-2} \big)$ and $ b_{k}^{(2)} = \frac{1}{\lambda - \mu } \big( \lambda^{k-1} -\mu^{k-1} \big)$, respectively. Using Binet's formulae we have the identity
$$
b_{k}^{(1)} b_{t}^{(2)} - b_{t}^{(1)} b_{k}^{(2)} = \bigg( \frac{\alpha_3}{\alpha_1} \bigg)^{t-2} b_{k-t+2}^{(1)},
$$
where $k \geq t$. Thus, we have
 \begin{eqnarray*}
\det(A)&=&
(1-a_{n+1})^{n-3}\bigg(\sum_{i=2}^{n-3}\sum_{j=i+1}^{n-2} \Delta(i,j) \big(\frac{\alpha_3}{\alpha_1}\big)^{n-j-1}b_{j-i+2}^{(1)}  \\
& & + \sum_{i=2}^{n-2} \Delta(i,(n-1)) b_{n-i+1}^{(1)}  + \sum_{i=2}^{n-2} \Delta(i,(n)) \frac{\alpha_1}{\alpha_3} b_{n-i+2}^{(1)} \\
& & + \Delta((n-1),(n)) \frac{\alpha_1}{\alpha_3} b_{3}^{(1)} \bigg).
\end{eqnarray*}
The proof follows from equalities
$$
\Delta(i,j) = a_{i-2}a_{j-1} - a_{i-1}a_{j-2},
$$
$$
\Delta(i,(n-1)) = (2 a_n - 1)a_{i} + (1-2 a_{n-1})a_{i+1} + (a_{n-1} - a_{n}) a_{i+2},
$$
$$
\Delta(i,(n)) = a_{i} + (1-2 a_{n})a_{i+1} + a_{n-1} a_{i+2},
$$
$$
\Delta((n-1),(n)) \frac{\alpha_1}{\alpha_3} b_{3}^{(1)} = 2a_n^2 - 2a_{n} - a_{n-1} + 1.
$$
\end{proof}

We cannot state that the determinant formula in Corollary~\ref{tribonaccidet} is elegant but it reduces an $n \times n$ determinant to a double sum.

\begin{cor}[\cite{Bozkurt}, Theorem 1]   \label{corollary}
Let $\{ U_k \}$ be the sequence defined by the recurrence relation given in (\ref{secondorderrecurrence}) with initial conditions $U_1=a,U_2=b$, $n>3$ and $A=circ(U_1, U_2, \ldots, U_n)$. Then
$$
\det(U)= (a^2-b U_n)(a-U_{n+1})^{n-2} + \sum_{k=2}^{n-1}(a U_{k+1}-b U_k) (a - U_{n+1})^{k-2} (q U_n -b + p a)^{n-k}.
$$
\end{cor}
\begin{proof}
Let $\{ a_k \}$ in Lemma~\ref{main} be the sequence $\{ U_k \}$ given in (\ref{secondorderrecurrence}) with initial conditions $U_1=a$ and $U_2=b$.  Then $\alpha_1 =a - U_{n+1}$, $\alpha_2 = b- p U_{1} - q U_{n}$ and hence $b_i^{(1)} = (-\alpha_2 / \alpha_1)^{i-1}$. Thus, by Lemma~\ref{main}, we have
\begin{eqnarray*}
  \det(A) &=& (a-U_{n+1})^{n-2} \sum_{k=2}^n \left|
                                                 \begin{array}{cc}
                                                   a & U_{(k)} \\
                                                   b & U_{(k+1)} \\
                                                 \end{array}
                                               \right| b_{n-k+1}^{(1)}
   \\
          &=& (a-U_{n+1})^{n-2} \big[(a^2-U_n b) + \sum_{k=2}^{n-1} \left|
                                                 \begin{array}{cc}
                                                   a & U_{k} \\
                                                   b & U_{k+1} \\
                                                 \end{array}
                                               \right| b_{n-k+1}^{(1)}\big] \\
   &=& (a-U_{n+1})^{n-2} \big[(a^2-b U_n) + \sum_{k=2}^{n-1} (a U_{k+1}- b U_k ) (-\frac{q U_n-b+pa}{a-U_{n+1}})^{n-k}\big].
\end{eqnarray*}
A simple calculation completes the proof.
\end{proof}
Renaming terms of sequence $\{U_k\}$ as $\{ W_{k-1} \}$ we obtain the same formula in Theorem~1 of Bozkurt's paper \cite{Bozkurt}. Also, by choosing convenient values for $p$, $q$, $a$ and $b$ in Corollary~\ref{corollary} we can obtain all determinant formulae in \cite{BozkurtTam,ShenCenHao}. Taking $(p,q,a,b)=(1,1,1,1)$, $(1,1,1,3)$, $(1,2,1,1)$ and $(1,2,1,3)$, we have Theorems 2.1 and 3.1 of \cite{ShenCenHao} and Theorems 2.1 and 2.2 of \cite{BozkurtTam}, respectively. Also, by Lemma~\ref{main}, we can easily evaluate the determinant of $A=circ(a,a^2,a^3, \ldots , a^n)$, where $a$ is a nonzero real number, as $\det(A)=a^n (1-a^n)^{n-1}$.

\end{document}